\documentclass{amsart}
\usepackage{color}
\usepackage{amssymb,amsmath,amsthm,amstext,amsfonts, amsrefs}
\usepackage[dvips]{graphicx}
\usepackage{psfrag}

\makeatletter \@addtoreset{equation}{section} \makeatother

\renewcommand\thetable{\thesection.\@arabic\c@table}

\theoremstyle{plain}
\newtheorem{maintheorem}{Theorem}

\newtheorem{theorem}{Theorem }[section]
\newtheorem{proposition}[theorem]{Proposition}
\newtheorem{lemma}[theorem]{Lemma}
\newtheorem{corollary}[theorem]{Corollary}

\theoremstyle{definition} \theoremstyle{remark}
\newtheorem{remark}[theorem]{Remark}

\newtheorem{question}{Question}

\newcommand{\htop}{h_{\topp}}

\newcommand{\be} {\beta}

\newcommand{\ep} {\epsilon}
\newcommand{\vep}{\varepsilon}

\newcommand{\la} {\lambda}

\newcommand{\Z}{\mathbb{Z}}
\newcommand{\N}{\mathbb{N}}

\newcommand{\topp}{\operatorname{top}}

\newcommand{\hW}{\widehat{\cW}}

\newcommand{\cW}{\mathcal{W}}

\usepackage{ifthen,tikz,,xkeyval,calc,xcolor}

\title{Entropy of partially hyperbolic flows with center dimension two}
\author{M. Roldan, R. Saghin and J. Yang}

\date{\today}

\begin{document}

\begin{abstract}
In this article we study the regularity of the topological and metric entropy of partially hyperbolic flows with two-dimensional center direction. We show that the topological entropy is upper semicontinuous with respect to the flow, and we give an example where the lower semicontinuity fails. We also show that if such a flow has no fixed points, then it is entropy expansive, and consequentely the metric entropy function is upper semicontinuous, there exist equilibrium states (and measures of maximal entropy), and principal symbolic extensions. 
\end{abstract}

\maketitle

\section{Statements and results}

The topological and metric entropy are important invariants in dynamical systems, they measure the complexity of the systems, and they were extensively studied during the last decades. On one hand one would like to understand the dependence of topological entropy with respect to perturbations of the system, and there are many studies in this direction, starting with works by Yomdin and Newhouse. On the other hand the dependence of the metric entropy with respect to the measure is important in the study of the pressure and equilibrium states, and it was extensively studied, starting with the work of Bowen. 

In general the entropy can fail to be continuous, however if we restrict our attention to systems which are smooth enough or have enough hyperbolicity, many positive results can be obtained. It is known the topological entropy is upper semicontinuous in the space of $C^{\infty}$ diffeomorphisms (\cite{Buzzi}), or $C^1$ diffeomorphisms away from tangency (\cite{LVY}). For these systems we also know that the metric entropy function is upper semicontinuous. In fact we know more, the $C^{\infty}$ maps are assymptotic entropy expansive, while the $C^1$ diffeomorphisms away from tangencies are entropy expansive. On the other hand the lower semicontinuity of the topological entropy holds for $C^{1+\alpha}$ diffeomorphisms on surfaces (\cite{Katok}).

Uniformly hyperbolic diffeomorphisms are very well behaved in this aspect, the topological entropy is locally constant, and they are expansive. We also have various results for partially hyperbolic diffeomorphisms with one dimensional center bundle. By \cite{LVY} (or previously \cite{DF}, \cite{DFPV}), they are entropy expansive, so the topological entropy and the metric entropy function are upper semicontinuous. For the known examples of partially hyperbolic diffeomorphisms with one dimensional center we know that the topological entropy is also lower semicontinuous, so in fact it is continuous in the $C^1$ topology (see \cite{SY}). In the case of partially hyperbolic diffeomorphisms with two dimensional center all the continuity properties of the entropy may fail. In the absence of a dominated splitting or some hyperbolicity of the center, one expects that $C^1$ generically there are persistent homoclinic tangencies, so the maps are not entropy expansive, and the metric entropy function is not upper semicontinuous (see \cite{DN,Asaoka,DF,DFPV,CT,BCF}). One can construct examples where the upper semicontinuity of the topological entropy fails in the $C^r$ topology (using homoclinic tangencies), and examples where the lower semicontinuity of the topological entropy fails even in the $C^{\infty}$ topology (see \cite{HSX} for example). On the other hand, it is worth mentioning that in the lack of any dominated splitting, the topological entropy is continuous $C^1$ generically (\cite{BCF}).

The goal of this paper is to extend these studies to the case of partially hyperbolic flows with two dimensional center bundle. In some sense, the partially hyperbolic flows with two dimensional center can be considered in between the partially hyperbolic diffeomorphisms with one and two dimensional center. Thus one expects that some of the properties that hold for partially hyperbolic diffeomorphisms with one dimensional center will also hold in our case, while other properties may fail. We will see that this is the case for the continuity properties of the entropy.

Let us comment that even though there are several results in the literature on the regularity of the entropy for partially hyperbolic diffeomorphisms, they do not contain the partially hyperbolic flows which we consider (or the time one map of these flows). Entropy expansiveness holds if the center bundle has a dominated splitting into one dimensional subbundles (\cite{DFPV}), or more generally it is far from tangencies (\cite{LVY}), or if all the Lyapunov exponents in the center direction for all invariant measures are nonpositive (or non-negative) (\cite{CY}). In \cite{Yang} the author obtains the entropy expansiveness and even the continuity of the topological entropy for Lorenz-like flows (including partially hyperbolic flows with 1D center). In general the time one map of a partially hyperbolic flow with 2 dimensional center does not fall in any of these categories.

Now we will recall some necessary notions and state our main results. Given a $C^1$ vector field $X\in \mathfrak{X}^1(M)$ on a compact Riemannian manifold $M$ we denote by $(\phi_t^X)_{t\in\mathbb R}$ the flow associated to $X$. If it is not necesary to mention $X$ we will write just $\phi_t$.

We say that a $D\phi_t$-invariant splitting $T_\Lambda M=E\oplus F$ is dominated if there are $C>0$ and $0<\la<1$ so that for all $t\ge 0$ and all $x\in \Lambda$, we have
\begin{equation}\label{eq:domination}
\frac{\|D\phi_t(x)\mid_{E_x}\|}{\|D\phi_{-t}(\phi_t(x))\mid_{F_{\phi_t(x)}}\|^{-1}} \le C \lambda^t.
\end{equation}

We say that a compact $(\phi_t)$-invariant set $\Lambda\subset M$ is partially hyperbolic if there exists a $D\phi_t$-invariant splitting $T_\Lambda M=E^s \oplus E^c \oplus E^u$, and constants $C>0$, $0<\lambda<1$ so that $\|D\phi_t(x)\mid_{E^s_x}\| \leq C\la^t$, $\|(D\phi_t(x)\mid_{E^u_x})^{-1}\| \leq C\la^t$ for all $x\in \Lambda$ and $t\ge 0$, and the decompositions $E^s\oplus E^c$ and $E^c \oplus E^u$ are dominated. If $\Lambda$ is the whole manifold $M$ then we say that $\phi_t$ is partially hyperbolic. In the definition of partial hyperbolicity one could allow that $E^s$ or $E^u$ be trivial, but not both. Moreover, notice that $E^c$ is trivial if and only if $\Lambda$ is a singular set since otherwise it contains the subspace generated by the vector field.

An equivalent definition of partial hyperbolicity would be that the time-$t$ map $\phi_t$ of the flow is partially hyperbolic, for any (or for some) $t>0$, in the classical sense of partially hyperbolic diffeomorphisms. It is known that the set of partially hyperbolic flows is $C^1$ open, and the invariant splitting depends continuously on the point and on the vectorfield generating the flow (again in the $C^1$ topology). It is known that the stable bundle $E^s$ and the unstable bundle $E^u$ integrate to stable and unstable foliations, $W^s$ and $W^u$. In general $E^c$ may not be integrable, but as we said it should always contain the direction of the flow $X$, so there exists at least a one dimensional foliation tangent to $E^c$ outside the singularities of $X$.

The topological entropy of a flow is equal to the topological entropy of the time one map of the flow, and satisfies the relation
$$
h_{top}(\phi)=h_{top}(\phi_1)=\frac 1{|t|}h_{top}(\phi_t)\hbox{ for every }t\neq 0.
$$

A Borel measure $\mu$ on $M$ is {\it invariant under $\phi$} if it is invariant under $\phi_t$ for any $t\in\mathbb R$. In particular it will be invariant under $\phi_1$, and the metric entropy of $\phi$ with respect to $\mu$ is 
$$
h_{\mu}(\phi)=h_{\mu}(\phi_1)=\frac 1{|t|}h_{\mu}(\phi_t)\hbox{ for every }t\neq 0.
$$
Even if it is not necesary that every $\phi_1$ invariant measure is also $\phi$ invariant, the variational principle still holds (see for example \cite{SV}):
\begin{eqnarray*}
h_{top}(\phi)&=&\sup_{\mu\in\mathcal M(\phi)}h_{\mu}(\phi)\\
&=&\sup_{\mu\in\mathcal M_{erg}(\phi)}h_{\mu}(\phi),
\end{eqnarray*}
where $\mathcal M(\phi)$ is the set of invariant probability measures of $\phi$, while $\mathcal M_{erg}(\phi)$ is the set of invariant ergodic probability measures of $\phi$.

Apart from the topological and metric entropies, we will consider other geometrically interesting quantity which is given by 
the geometric growth on invariant foliations. Given an $\phi_t$-invariant foliation $W$ on $M$ with $C^1$ 
leaves, let $W_r(x)$ denote the ball of radius $r$ around $x$ in the leaf $W(x)$ that contains $x$ 
(the ambient Riemannian structure induces a Riemannian structure on $W(x)$, this provides $W(x)$ with a metric). 
The \emph{absolute geometric growth of the foliation $W$ under the flow $X$} is defined in the following way, addapting the definitions for diffeomorphisms introduced in \cite{HSX}, \cite{Saghin}:
 $$
 \chi_W(\phi) = \limsup_{t\to\infty} \frac1t \log (\sup_{x\in M}vol (\phi_t(W_r(x)))).
 $$
The volume growth is clearly independent of $r>0$. In the case of a partially hyperbolic flow $\phi_t$ we denote by $\chi_u(\phi)$ (respectively $\chi_s(\phi)$) the geometric growth along the strong unstable foliation for $\phi_t$ (respectively the stable foliation for $\phi_{-t}$).

Our first result concerns the upper semicontinuity of the topological entropy function. In this case some of the results on partially hyperbolic diffeomorphisms with 1D center can be extended to partially hyperbolic flows with 2D center. Let us denote by $\mathfrak{X}^r_2(M)$ the set of $C^r$ vector fields on $M$ which generate partially hyperbolic flows with center dimension 2.

\begin{maintheorem}\label{thm:topusc}
The following hold:
\begin{enumerate}
\item If $X\in\mathfrak{X}^1_2(M)$, then $\htop(\phi_1)=\max\{\chi_u(\phi), \chi_s(\phi)\}$.
\item $X \mapsto \htop(\phi^X)$ is upper-semicontinuous on $\mathfrak{X}^1_2(M)$, the set of $C^1$ partially hyperbolic flows with 2-dimensional center bundle.
\item If $X\in\mathfrak{X}^1_2(M)$, and all the central Lyapunov exponents of all the invariant measures of $\phi_1$ have the same sign, then $f \mapsto \htop(f)$ is upper-semicontinuous at $\phi_1$ on the set of $C^1$ diffeomorphisms.
\end{enumerate}
\end{maintheorem}

Let us remark that part (3) of the theorem falls whithin the hypothesis of \cite{CY}, so it is not a new result. Since the proof follows easily from our other results, we will include it here in our considerations.

An interesting question is what happens if we drop the condition on Lyapunov exponents in part (3) of Theorem \ref{thm:topusc}.
\begin{question}
Suppose that $X\in\mathfrak{X}^1_2(M)$, is it true that $f \mapsto \htop(f)$ is upper-semicontinuous at $\phi_1$ on the set of $C^1$ diffeomorphisms? How about if we also assume that $\phi$ has no fixed points?
\end{question}

The main method used in the literature in order to get the upper semicontinuity of the topological entropy (or the metric entropy function) is to obtain first entropy expansiveness, then apply some standard arguments going back to Bowen \cite{Bowen}. We preffer to give a different proof of the upper semicontinuity of the topological entropy based in the study of the volume growth, proof which may be useful in other situations when the center direction of a partially hyperbolic diffeomorphism has low dimension or good control of the expansion/contraction rates. For example it is straightforward to show that Theorem \ref{thm:topusc} holds in a space of partially hyperbolic diffeomorphisms with center dimension $d$ and with the property that every invariant measure of high entropy has $d-1$ zero exponents. Our proof is based on the following lemma which may have its own interest.

\begin{lemma}\label{l:volusc}
The maps $f\mapsto \chi_u(f)$ and $f\mapsto \chi_s(f)$ are upper semicontinuous on the space of $C^1$ partially hyperbolic diffeomorphisms.
\end{lemma}

We remark that in the case of hyperbolic flows, one not only obtains the continuity of the topological entropy at the time one map of the flow, but this continuity also propagates to a $C^1$ neighborhood of the time one map of the flow (see \cite{SY}). However when the center of the partially hyperbolic flow is two dimensional this may not be the case, and the topological entropy may not be upper semicontinuous on a neighborhood of the time one map of the flow. We have the following result.

\begin{maintheorem}\label{thm:notnear}
There exists $X\in\mathfrak{X}^{\infty}_2(M)$ such that, for any $r\geq 1$, $C^r$ arbitrarily close to $\phi_1$ there exists a map $f_0$ such that $f \mapsto \htop(f)$ is not upper-semicontinuous at $f_0$ on the set of $C^r$ diffeomorphisms.
\end{maintheorem}

In the case of the lower-semicontinuity of the topological entropy, the parallel between partially hyperbolic maps with 1D center and partially hyperbolic flows with 2D center fails. The map $X \mapsto \htop(\phi^X)$ is not lower-semicontinuous on the set of $C^{\infty}$ vector fields generating partially hyperbolic flows with 2-dimensional center bundle.

\begin{maintheorem}\label{thm:notoplsc}
There exists $X_0\in\mathfrak{X}^{\infty}_2(M)$, such that the map $X \mapsto \htop(\phi^X)$ is not lower-semicontinuous on $\mathfrak{X}^{\infty}_2(M)$ at $X_0$.
\end{maintheorem}

This example will also show that the maps $X\mapsto\chi_u(\phi^X)$ and $X\mapsto\chi_s(\phi^X)$ are not lower semicontinuous on $\mathfrak{X}^{\infty}_2(M)$ at $X_0$.

There are situations where we can guarantee the lower semicontinuity at some vector field $X_0$. This happens if we have a bit of hyperbolicity in the center bundle.

\begin{proposition}\label{prop:LSC}
Suppose that $X\in\mathfrak{X}^2_2(M)$, and suppose that the following holds:
\begin{equation}\label{eq:lsc}
\htop(\phi^{X_0})>\min\{\chi_u(\phi^{X_0}), \chi_s(\phi^{X_0})\}.
\end{equation}
Then the map $X \mapsto \htop(\phi^X)$ is lower-semicontinuous at $X_0$ on $\mathfrak{X}^1_2(M)$.
\end{proposition}

\begin{remark}
A similar result holds for partially hyperbolic flows with 3 dimensional center bundle. Suppose that $X_0\in\mathfrak{X}^2_3(M)$, the $C^2$ vector fields generating flows which are partially hyperbolic with 3 dimensional center bundle. If $\htop(\phi^{X_0})>\max\{\chi_u(\phi^{X_0}), \chi_s(\phi^{X_0})\}$, then the map $X \mapsto \htop(\phi^X)$ is lower-semicontinuous at $X_0$ on $\mathfrak{X}^1_3(M)$.
\end{remark}

\begin{remark}
The condition \eqref{eq:lsc} can be replaced by the following condition: there exists a measure of maximal entropy of $\phi^{X_0}$ (or a sequence of invariant measures with the metric entropies converging to $\htop(\phi^{X_0})$) with one nonzero center exponent.
\end{remark}

Regarding the upper-semicontinuity of the metric entropy function, or more generally the entropy expansiveness, the sitation of the partially hyperbolic flows with 2D center is quite similar to the case of partially hyperbolic diffeomorphisms with 1D center. For some $\vep>0$, $x\in M$ denote de infinite dynamical ball  at $x$ by 
\[B_{\infty}(x,\vep)=\bigcap_{\Z\ni n}\phi_{-n}B_{\vep}(\phi_n(x))= \{y\in M\,\,|\,\, d(\phi_n(x),\phi_n(y))\leq\vep\text{ for all } n\in\Z\}\]
We say that $\phi_1$ is {\it entropy expansive} if there exists $\vep>0$ such that $\htop(\phi_1|{B_{\infty}(x,\vep)})=0$ for all $x\in M$.

\begin{maintheorem}\label{thm:metrusc}
Assume that $X\in\mathfrak{X}^1_2(M)$, and $\phi$ does not have fixed points. Then $\phi_1$ is entropy expansive, and consequently the map $\mu\mapsto h_{\mu}(\phi)$ is upper-semicontinuous on the space $\mathcal M(\phi)$ of invariant Borel probability measures for $\phi$.
\end{maintheorem}

As a consequence we obtain the existence of equilibrium states and measures of maximal entropy:

\begin{corollary}\label{c:estates}
Assume that $X\in\mathfrak{X}^1_2(M)$, and $\phi$ does not have fixed points. Then every continuous potential on $M$ has equilibrium states. In particular there exist measures of maximal entropy.
\end{corollary}

It is also  well known that a diffeomorphism which is entropy expansive (or more generally asymptotically entropy expansive) admits a principal symbolic extension (see \cite{BFF}). Since we proved above that $\phi_1$ is entropy expansive, we also obtain the following corollary:

\begin{corollary}\label{c:pse}
Assume that $X\in\mathfrak{X}^1_2(M)$, and $\phi$ does not have fixed points. Then $\phi_1$ admits a principal symbolic extension.
\end{corollary}

In order to obtain the entropy expansiveness, we have to make the technical assumption on the lack of the singularities of the vector field. An interesting question is weather this condition can be removed.

\begin{question}
Assume that $X\in\mathfrak{X}^1_2(M)$, possible with singularities, is it true that $\phi_1$ is entropy expansive? How about assymptotic entropy expansive? How about if we assume that $\phi_1$ is dynamically coherent?
\end{question}

Another interesting question is weather a similar result can be obtained for dynamically coherent partially hyperbolic diffeomorphisms with 2D center, assuming that the center foliation contains a 1D sub-foliation. Our proof uses some kind of strong uniformity of the dynamics along this 1D sub-foliation. Maybe an easier question is the following.

\begin{question}
Let $f$ be a partially hyperbolic diffeomorphism with 2D center, which is also a skew product with circle fibers over a dynamically coherent partially hyperbolic diffeomorphism with 1D center. Is it true that $f$ is (assymptotic) entropy expansive?
\end{question}

In Section \ref{s:tusc} we will consider the upper semicontinuity of the topological entropy and of the volume growth, and we will prove Lemma \ref{l:volusc}, Theorem \ref{thm:topusc} and Theorem \ref{thm:notnear}. In Section \ref{s:lsc} we will show the lack of lower semicontinuity (Theorem \ref{thm:notoplsc}) and we will prove Proposition \ref{prop:LSC}, while in the last section we will treat the entropy expansiveness and the upper semicontinuity of the metric entropy function (Theorem \ref{thm:metrusc}).

%%%%%%%%%%%%%%%%%%%%%%%%%%%%%%%%%%%%%%%%%%%%%%%%%%%

\section{Upper semicontinuity of the topological entropy}\label{s:tusc}

We start this section by proving that the geometric growth along the foliations can be computed using any time-$t$ map for the flow, and the geometric growth of a foliation under a flow coincides with the geometric growth of the foliation under the time one map of the flow.

\begin{proposition}\label{prop:A}
Let $X\in\mathfrak{X}^1(M)$ be a $C^1$-vector field generating a partially hyperbolic flow $(\phi_t)$ wich leaves invariant a foliation with $C^1$ leaves $W$. Then,
$$
\chi_W(\phi)=\chi_W(\phi_1)=\frac{1}{t}\chi_W(\phi_t)\hbox{ for every }t>0.
$$
\end{proposition}

\begin{proof}[Proof of Proposition~\ref{prop:A}]
Given $t>0$, for any $n\in\N$ write $n=tk(n)+l(n)$, for some $k=k(n)\in\N$ and $0\leq l(n)<t$. Notice that $k(n)\to\infty$ as $n\to\infty$. For any $C^1$ disk $A$ in $M$ we clearly have
$$
vol(\phi_l(A)\leq vol(A)\cdot\|D\phi_l\|^{\dim(M)}.
$$
Let $C_t=\sup_{l\in[-t,t]}\|D\phi_l\|^{\dim(M)}$. Then
$$
\frac 1n\log\sup_{x\in M}vol(\phi_n(W_r(x)))=\frac 1{tk(n)+l(n)}\log\sup_{x\in M}vol(\phi_{l(n)}(\phi_t^{k(n)}(W_r(x)))),
$$
then furthermore

\begin{eqnarray*}
\frac 1{t(k+1)}\log\frac 1{C_t}\sup_{x\in M}vol(\phi_t^{k}(W_r(x)))&\leq& \frac 1n\log\sup_{x\in M}vol(\phi_n(W_r(x)))\\
&\leq&\frac 1{tk}\log C_t\sup_{x\in M}vol(\phi_t^{k}(W_r(x))).
\end{eqnarray*}

Taking the limit when $n$ and $k(n)$ go to infinity, we obtain
$$
\chi_W(\phi)=\frac 1t\chi_W(\phi_t).
$$

\end{proof}

\subsection{Proof of Lemma \ref{l:volusc}}

In this section we will prove the Lemma \ref{l:volusc}. It can be generalized for other foliations, not necesarily the unstable one; the unique requirement is that the leaves depend continuously on the map in the $C^1$ topology, and they are uniformly expanding.

\begin{proof}[Proof of Lemma \ref{l:volusc}]

Let $f$ be a $C^1$ partially hyperbolic diffeomorphism, and let

$$
\chi_u(f)=\limsup_{n\rightarrow\infty}\frac 1n\log\sup_{x\in M}vol(f^n(W^u_r(x))).
$$
This means that for any $\ep>0$ there exists $n_0\in\N$ (which we can assume arbitrarily large) such that for any $x\in M$ we have
$$
vol(f^n(W_r^u(x)))<e^{n(\chi_u(f)+\ep)}\hbox{ for any }n\geq n_0.
$$
Since $(x,f)\mapsto W_r^u(x,f)$ is continuous in the $C^1$ topology, there exists a $C^1$ neighborhood $\mathcal U$ of $f$ such that for any $g\in\mathcal U$ the following holds:
$$
vol(g^{n_0}(W_r^u(x)))<e^{n_0(\chi_u(f)+2\ep)}\hbox{ for any }x\in M.
$$
Since $W^u$ is uniformly expanding, if $n_0$ is sufficiently large, we have that, for all $x\in M$ and $g\in\mathcal U$, $g^{-n_0}(W^u_{r/2}(x))\subset W^u_r(g^{-n_0}(x))$. Furthermore there exists a constant $D_r>0$ such that, for any $x\in M$, the ball $W^u_{2r}(x)$ can be covered by $D_r$ other balls of radius $r$ inside the foliation $W^u$. Let $C_r$ be a lower bound for the volumes of the balls of radius $r/2$ inside $W^u$, and let $K_r$ be an upper bound for the volumes of the balls of radius $r$ inside $W^u$, for all $g\in\mathcal U$.

Given $g\in\mathcal U$, $x\in M$, let $S(x,g)$ be a maximal $r$-separated set in $g^{n_0}(W^u_r(x))$ (in the metric induced on the leaves). Then we know that the balls $W_{r/2}^u(y)$, $y\in S(x,r)$ are mutually disjoint. Given an $y\in S(x,r)\subset g^{n_0}(W^u_r(x))$, we get that $g^{-n_0}(W^u_{r/2}(y))\subset W^u_r(g^{-n_0}(y))\subset W^u_{2r}(x)$. Then
$$
\cup_{y\in S(x,g)}W_{r/2}^u(y)\subset g^{n_0}(W^u_{2r}(x)),
$$
and from here and the uniform upper bound on the growth of the volume of balls of radius $r$, we get a bound for the cardinality of $S(x,r)$:
$$
\# S(x,r)\leq \frac{D_r}{C_r}e^{n_0(\chi_u(f)+2\ep)}, \ \ \ \forall x\in M.
$$

On the other hand $g^{n_0}(W^u_r(x))$ can be covered by $\# S(x,r)$ new balls of radius $r$:
$$
g^{n_0}(W^u_r(x))\subset\cup_{y\in S(x,g)}W^u_r(y).
$$

By induction on $k$ one can then show that $g^{kn_0}(W^u_r(x))$ can be covered by at most $\left(\frac{D_r}{C_r}e^{n_0(\chi_u(f)+2\ep)}\right)^k$ new balls of radius $r$, and in particular

$$
vol(g^{kn_0}(W^u_r(x)))\leq K_r\frac{D_r^k}{C_r^k}e^{kn_0(\chi_u(f)+2\ep)},
$$
for  any $k\in\mathbb N$, $x\in M$, $g\in\mathcal U$.

Let $C$ be a bound for the derivative in $\mathcal U$, then for any $g\in\mathcal U$ we have
\begin{eqnarray*}
\chi_u(g)&=&\limsup_{n\rightarrow\infty}\frac 1n\log\sup_{x\in M}vol(g^n(W^u_r(x)))\\
&=&\limsup_{n\rightarrow\infty}\frac 1{k(n)n_0+l(n)}\log\sup_{x\in M}vol(f^{k(n)n_0+l(n)}(W^u_r(x)))\\
&\leq&\limsup_{n\rightarrow\infty}\frac 1{k(n)n_0}\log\sup_{x\in M}vol(f^{k(n)n_0}(W^u_r(x)))\cdot C^{n_0\dim M}\\
&\leq&\limsup_{n\rightarrow\infty}\left(\frac{\dim M\log C}{k(n)}+\frac{\log K_r}{k(n)n_0}+\frac{\log D_r-\log C_r}{n_0}+\chi_u(f)+2\ep  \right)\\
&=&\chi_u(f)+2\ep+\frac{\log D_r-\log C_r}{n_0}.
\end{eqnarray*}

Since $n_0$ can be taken arbitrarily large and $\ep$ can be taken arbitrarily small at the expense of making $\mathcal U$ smaller, and $C_r, D_r$ are bounded, we obtain the upper continuity of $\chi_u$ with respect to the diffeomorphism. The upper semicontinuity of $\chi_s$ is proven in a similar way, using $f^{-1}$.

\end{proof}

%%%%%%%%%%%%%%%%%%%%%%%%%%%%%%%%%%%%%%%%%%%%%%%%%

\subsection{Proof of Theorem~\ref{thm:topusc}}

This proof is an immediate consequence of Lemma \ref{l:volusc} and the results from \cite {Saghin}, \cite{HSX}. Let $X\in\mathfrak{X}^1(M)$ be a $C^1$-vector field generating a partially hyperbolic flow $(\phi_t)$ with invariant splitting $TM=E^s\oplus E^c \oplus E^u$, so the time-1 map $\phi_1$ is also partially hyperbolic.

Firstly it is well known that $\htop(\phi_1)\ge \chi_u(\phi_1)=\chi_u(\phi)$, and also $\htop(\phi_{1})=\htop(\phi_{-1})\ge \chi_s(\phi)$ (see \cite{SX1} or \cite{Saghin} for example).
From the previous consideration we have
\begin{equation}\label{eq:SX}
\htop(\phi)\ge \max\{\chi_u(\phi), \chi_s(\phi)\}.
\end{equation}

Secondly, we know from \cite{Saghin} that if $\nu$ is a $\phi_1$-invariant and ergodic probability measure then
\begin{equation}\label{eq:Su}
h_{\nu}(\phi_1)
	\leq \chi_u(\phi_1) + \max \lambda^c\dim(E^{cs}).
\end{equation}
where $\lambda^c$ denotes the Lyapunov exponents of $\nu$ in the direction $E^c$. A similar formula can be obtained for $\phi_{-1}$ and $\chi_s(\phi)$:
\begin{equation}\label{eq:Ss}
h_{\nu}(\phi_1)=h_{\nu}(\phi_{-1})
	\leq \chi_s(\phi_1) - \min\lambda^c\dim(E^{cu}).
\end{equation}

If $\nu$ is supported on a fixed point, then clearly $h_{\nu}(\phi_1)=0\leq\chi_u(\phi)$. If $\nu$ is not supported at a fixed point, since it is ergodic, it must contain in its support a point $p\in M$ such that $X(p)\neq 0$. Take a neighborhood $U$ of $p$ such that $m<\|X|_U\|<M$, for some $m,M>0$. Poincar\'e recurrence guarantees that there exists a generic point $q$ of $\nu$ inside $U$ which returns to $U$ infinitely many times under iterates of $\phi_1$. Then
\begin{eqnarray*}
\lambda(\phi_1, X(q))&=&\lim_{n\rightarrow\infty}\frac 1n\log \|D\phi_n(q)\cdot X(q)\|\\
&=&\lim_{n\rightarrow\infty}\frac 1n\log\|X(\phi_n(q)\|\\
&=&0.
\end{eqnarray*}

This shows that one center Lyapunov exponent of $\nu$ must be zero. If $X\in\mathfrak{X}^1_2(M)$, then either $\max \lambda^c=0$ or $\min \lambda^c=0$, so from \ref{eq:Su}, \ref{eq:Ss} we get that $h_{\nu}(\phi)$ is bounded by above by either $\chi_u(\phi)$ or $\chi_s(\phi)$. Together with \ref{eq:SX}, and using the variational principle, we obtain that
$$
\htop(\phi)=\max\{\chi_u(\phi), \chi_s(\phi)\},
$$
which shows the part (1) of the Theorem.

Part (2) follows immediately from the part (1) and from the Lemma \ref{l:volusc}.

Suppose that the conclusion (3) does not hold, then there exists a sequence of diffeomorphisms $f_n$ converging to $\phi_1$ in the $C^1$ topology such that $\htop(f_n)>\htop(\phi)+\delta$ for some $\delta>0$. From the variational principle we have a sequence of probability measures $\mu_n$ invariant for each $f_n$ such that
$$
h_{\mu_n}(f_n)>\htop(\phi)+\delta,\ \ \forall n\geq 1.
$$
Eventually passing to a subsequence, we may assume that $\mu_n$ converges in the weak star topology to a measure $\mu$, which by a standard argument has to be an invariant probability for $\phi_1$.

Since we know from Lemma \ref{l:volusc} that the stable and unstable volume growth are upper semicontinuous with respect to the map, we can assume that
\begin{equation}\label{eq:vgbound}
\chi_u(f_n)<\chi_u(\phi)+\frac {\delta}3,\ \ \chi_s(f_n)<\chi_s(\phi)+\frac {\delta}3,\ \ \forall n\geq 1.
\end{equation}

Without loss of generality let us assume that the central exponents of $\mu$ for $\phi_1$ are nonnegative, otherwise we can do the same considerations for $\phi_{-1}$. Then for $\mu$ almost every $x\in M$ we have
$$
\lim_{k\rightarrow\infty}\frac 1k\log\|D\phi_1^k(x)|_{E^c(x)}\|=0.
$$
Since $\frac 1k\log\|D\phi_1^k(x)|_{E^c(x)}\|$ is uniformly bounded, we get that
$$
\lim_{k\rightarrow\infty}\int_M\frac 1k\log\|D\phi_1^{k_0}(x)|_{E^c(x)}\|d\mu=0,
$$
so there exists $k_0$ such that
$$
\int_M\frac 1{k_0}\log\|D\phi_1^{k_0}(x)|_{E^c(x)}\|d\mu<\frac {\delta}{3\dim(E^{cs})}.
$$
Because $\mu_n$ converges weakly to $\mu$, there exists $n_0$ such that
$$
\int_M\frac 1{k_0}\log\|D\phi_1^{k_0}(x)|_{E^c(x)}\|d\mu_n<\frac {\delta}{3\dim(E^{cs})},\ \ \forall n\geq n_0.
$$
Since $f_n$ converges to $\phi_1$ in the $C^1$ topology, and $E^c$ is continuous with respect to the map, we have that $\|Df_n^{k_0}(x)|_{E^c(x)}\|$ converges uniformly to $\|D\phi_1^{k_0}(x)|_{E^c(x)}\|$, so there exists $n_1>n_0$ such that
$$
\left|\frac 1{k_0}\log\|D\phi_1^{k_0}(x)|_{E^c(x)}\|-\frac 1{k_0}\log\|Df_{n_1}^{k_0}(x)|_{E^c(x)}\|\right|<\frac {\delta}{3\dim(E^{cs})}.
$$
Then
$$
\int_M\frac 1{k_0}\log\|Df_{n_1}^{k_0}(x)|_{E^c(x)}\|d\mu_{n_1}<\frac {2\delta}{3\dim(E^{cs})}.
$$
Estimating the central exponents of $\mu_{n_1}$ for $f_{n_1}$, and using the subadditivity of $\log\|Df_{n_1}^{k_0}(x)|_{E^c(x)}\|$, we get
\begin{eqnarray*}
\max\lambda_i^c(\mu_{n_1},f_{n_1})&=&\lim_{l\rightarrow\infty}\int_M\frac 1{lk_0}\log\|Df_{n_1}^{lk_0}(x)|_{E^c(x)}\|d\mu_{n_1}\\
&\leq&\int_M\frac 1{k_0}\log\|Df_{n_1}^{k_0}(x)|_{E^c(x)}\|d\mu_{n_1}\\
&<&\frac {2\delta}{3\dim(E^{cs})}.
\end{eqnarray*}
In other words the upper central exponent is upper semicontinuous. Using the formulas \ref{eq:Su} and \ref{eq:vgbound} we then obtain
$$
h_{\mu_{n_1}}(f_{n_1})\leq\chi_u(f_{n_1})+\max\lambda_i^c(\mu_{n_1},f_{n_1})\cdot\dim(E^{cs})<\chi_u(\phi)+\delta\leq\htop(\phi)+\delta,
$$
which is a contradiction. This finishes the proof of (3).

\subsection{No upper semicontinuity in a neighborhood}

In this subsection we will prove Theorem \ref{thm:notnear}.

\begin{proof}[Proof of Theorem \ref{thm:notnear}]

Let $\alpha_t$ be the suspension flow of a $C^{\infty}$ hyperbolic diffeomorphism $A$ of the torus $\mathbb T^2$, with constant roof function one, defined on the compact suspension manifold $N$, let $M=N\times\mathbb T^1$, and let $\phi$ be the product of $\alpha$ with the identity flow on $\mathbb T^1$. Clearly $\phi$ is a $C^{\infty}$ partially hyperbolic flow with two dimensional center bundle, and the topological entropy of $\phi_1$ is equal to the topological entropy of $A$. We can find a neighborhood of $\mathbb T^2\times\{0\}\times\mathbb T^1$, which is $C^{\infty}$ diffeomorphic to $K:=\mathbb T^2\times [-\delta,\delta]\times\mathbb T^1$, such that $\alpha_1$ written in this chart is of the form
$$
\alpha_1(x,y,z)=(Ax,y,z),\ \ \forall x\in\mathbb T^2,\ y\in[-\delta,\delta],\ z\in\mathbb T^1.
$$
In other words $\phi_1$ restricted to $K$ is the product between $A$ on $\mathbb T^2$ and the identity on $[-\delta,\delta]\times\mathbb T^1$.

For any $r$, arbitrarily $C^r$ close to the identity on $[-\delta,\delta]\times\mathbb T^1$ there exists a $C^{\infty}$ diffeomorphism $g_0$ wich has a hyperbolic periodic point with a homoclinic loop inside $[-\delta,\delta]\times\mathbb T^1$, and is the identity on a neighborhood of the boundary. This can be done for example by making $g_0$ the time one map of a Hamiltonian flow generated by a smooth Hamiltonian of the form
$$
H(y,z)=\ep_1\rho\left(\frac{d((y,z),(0,0))}{4\delta}\right)+\ep_2y\rho\left(\frac y{2\delta}\right),
$$
for some smooth bump function $\rho$ wich is 1 on $[0,1/2]$ and has the support on $[0,1]$, while $\ep_1$ and $\ep_2$ are small enough. The only ergodic invariant measures of the Hamiltonian  flow are supported on periodic orbits and fixed points (a saddle, a center and degenerate fxed points), so the topological entropy of $g_0$ is zero.

Let $f_0$ be a perturbation of $\phi_1$ such that $f_0=A\times g_0$ on $K$, while outside $K$ we leave $f_0=\phi_1$. Clearly $f_0$ is $C^r$ close to $\phi_1$. We compute the topological entropy of $f_0$:
\begin{eqnarray*}
\htop(f_0)&=&\max\{\htop(f_0|_{K}),\htop(f_0|_{M\setminus K}\}\\
&=&\max\{\htop(A\times g_0|_{K}),\htop(\phi_1|_{M\setminus K}\}\\
&=&\max\{\htop(A)+\htop(g_0),\htop(\phi_1|_{M\setminus K}\}\\
&=&\htop(A),
\end{eqnarray*}
since $\htop(\phi_1|_{M\setminus K})\leq\htop(\phi_1)=\htop(A)$.

Having a homoclinic loop is similar to having a homoclinic tangency of arbitrarily high order, and the standard methods of "snake-like" perturbations developed in \cite{DN} for example show that the topological entropy is not lower semicontinuous at $g_0$. In fact there exists a sequence of $C^{\infty}$ diffeomorphisms $g_n$, $n>0$, converging to $g_0$ in the $C^r$ topology, and there exists an $\lambda>0$ such that $\htop(g_n)>\lambda$ for all $n>0$.

Now we can define the sequence of $C^{\infty}$ diffeomorphisms $f_n$ on $M$ as $f_n=A\times g_n$ on $K$ and $f_n=f_0=\phi_1$ outside $K$. Clearly $f_n$ converges to $f_0$ in the $C^r$ topology. Also
\begin{eqnarray*}
\htop(f_n)&\geq&\htop(f_n|_{K})\\
&=&\htop(A\times g_n)\\
&=&\htop(A)+\htop(g_n)\\
&>&\htop(A)+\lambda\\
&=&\htop(f_0)+\lambda,
\end{eqnarray*}
which shows that the topological entropy is not upper semicontinuous at $f_0$ on the space of $C^r$ diffeomorphisms.

\end{proof}

%%%%%%%%%%%%%%%%%%%%%%%%%%%%%%%%%%%%%%%%%%%%%%%%%

\section{No lower semicontinuity}\label{s:lsc}

In this section we start by proving Theorem \ref{thm:notoplsc}. The example is somehow similar to the example in \cite{HSX}.

\begin{proof}[Proof of Theorem \ref{thm:notoplsc}]

Let $\alpha_t$ be a $C^{\infty}$ hyperbolic flow on the compact Riemannian manifold $N$, generated by the vector field $X$ on $N$. Let $M=N\times\mathbb T^1$, and $a:\mathbb T^1\rightarrow (0,\infty)$ a nonconstant $C^{\infty}$ function. Let $\phi$ be the flow on $M$ given by
$$
\phi_t(x,s)=(\alpha_{a(s)t}(x),s),\ \ \forall t\in\mathbb R,\ \ \forall x\in N,\ \ \forall s\in\mathbb T^1.
$$
Clearly $\phi$ is a partially hyperbolic flow generated by the vector field $Y(x,s)=(a(s)X(x),0)$ on $N$, the center bundle is two dimensional, satisfying $E^c(x,s)=E^c(x)\times T\mathbb T^1$. $M$ is subfoliated by the invariant submanifods $N\times\{s\}$, $s\in\mathbb T^1$, and on each such submanifold the topological entropy is
$$
\htop(\phi|_{N\times\{s\}})=a(s)\cdot\htop(\alpha),
$$
so we get that
$$
\htop(\phi)=\max_{s\in\mathbb T^1}a(s)\cdot\htop(\alpha).
$$

Let $Y^r$ be the $C^{\infty}$ family of $C^{\infty}$ vector fields on $M$ given by $Y^r(x,s)=(a(s)X(x),r)$, for any $r\in(-1,1)$, and let $\phi^r$ be the $C^{\infty}$ family of $C^{\infty}$ flows generated by $Y^r$. Clearly $\phi_0=\phi$. One can easily check that
$$
\phi^r_t(x,s)=(\alpha_{\int_0^ta(s+ur)du}(x),s+tr).
$$
In particular $\phi^r$ are all partially hyperbolic with the same center bundle as $\phi$. Observe that
$$
\phi^r_{\frac 1r}(x,s)=(\alpha_{\int_0^{\frac 1r}a(s+ur)du}(x),s+1)=(\alpha_{\frac 1r\int_{\mathbb T^1}a(v)dv}(x),s),
$$
i.e. $\phi^r_{\frac 1r}$ is the product of $\alpha_{\frac 1r\int_{\mathbb T^1}a(v)dv}$ on $M$ with the identity on $\mathbb T^1$. Calculating the topological entropy of $\phi^r$ for $r\neq 0$ we obtain
\begin{eqnarray*}
\htop(\phi^r)&=&r\htop(\phi^r_{\frac 1r})\\
&=&r\htop(\alpha_{\frac 1r\int_{\mathbb T^1}a(v)dv})\\
&=&\int_{\mathbb T^1}a(v)dv\cdot\htop(\alpha).
\end{eqnarray*}
Since $a$ is not constant, for $r\neq 0$ we have
$$
\htop(\phi^r)=\int_{\mathbb T^1}a(v)dv\cdot<\max_{s\in\mathbb T^1}a(s)=\htop(\phi^0),
$$
so $\htop(\phi^r)$ is not lower semicontinuous at $r=0$.

\end{proof}

Next we will show Proposition \ref{prop:LSC}.

\begin{proof}[Proof of Proposition \ref{prop:LSC}]

Assume that  $\htop(\phi^{X_0})> \min\{\chi_u(\phi^{X_0}), \chi_s(\phi^{X_0})\}$. Without loss of generality less us assume that $\htop(\phi^{X_0})>\chi_u(\phi^{X_0})$. Using the invariance principle, we can choose $\mu$ an invariant ergodic probability for $\phi^{X_0}$ such that $h_\mu(\phi^{X_0})>\chi_u(\phi^{X_0})$, and $h_\mu(\phi^{X_0})>\htop(\phi^{X_0})-\ep$, for $\ep$ arbitrarily close to zero. Recall that \eqref{eq:Su} gives us
$$
h_\mu(\phi^{X_0})\leq \chi_u(\phi^{X_0}) + \max{\lambda^c} \dim(E^{cs}).
$$
This implies that $\mu$ has a center exponent strictly greater than zero, so it is a hyperbolic invariant measure. Then the lower semicontinuity of the topological entropy will follow from the following continuous time Katok's approximation lemma due to Lian and Young (this is where we need the $C^2$ condition on $X_0$).

\begin{theorem} \cite[Theorem~D']{LY12}\label{LianYoung}
Take $\phi_1$ be the time one map of a $C^2$ flow and let $\mu$ be a $\phi_1$-invariant hyperbolic probability measure so that $h_\mu(\phi_1)>0$. Then, given $\ep>0$ there exists a smooth transversal $\mathcal D$ to the flow, a return map $T : \mathcal D \to \mathcal D$ and a bi-invariant horseshoe $\Omega\subset \mathcal D$ so that $\htop(\phi_1\mid_{\hat \Omega}) \ge h_\mu(\phi_1) -\ep$, where $\hat\Omega=\bigcup_{t\in \mathbb R} \phi_t(\Omega)$.
\end{theorem}

Since an invariant horseshoe is persistent among nearby flows, and the entropy restricted to the horseshoe is continuous with respect to the flow (see \cite{Contreras,KKPW}), the lower semicontinuity of the topological entropy at $X_0$ follows.

\end{proof}

%%%%%%%%%%%%%%%%%%%%%%%%%%%%%%%%%%%%%%%%%%%%%%%%%

\section{Entropy expansiveness and upper semicontinuity of the entropy function}

In this section we will prove Theorem \ref{thm:metrusc}. We will only show the entropy expansiveness, because the upper semicontinuity of the metric entropy function and Corollaries \ref{c:estates} and \ref{c:pse} are standard consequences.

As we said before, the central distribution may not be, in general, integrable. The knowledge of the existence of such center manifolds would simplify considerably the proof. Since we do not necesarily have center leaves, we need to  use the concept of "fake'' foliation. It turns out that such center fake foliations are suficient for our purpose. The following proposition is a standard tool in this area, for a proof see \cite{BW} for example.

\begin{proposition}\label{p:fakefol}
Let $f\colon M\to M$ be a $C^1$ partially hyperbolic diffeomorphism with decomposition $TM=E^u\oplus E^c\oplus E^s$. 
For any $\vep>0$, there exist constants $r_1>0$, $C>0$ such that any neighborhood $B(p,r_1)$ is foliated by foliations 
$\hW^u_p$, $\hW^c_p$, $\hW^s_p$, $\hW^{cu}_p$ and $\hW^{cs}_p$ with the following properties, for each $\be\in\{u,s,c,cs,cu\}$:
\begin{itemize}
\item [(i)] Almost tangency to invariant distributions. For each $q\in B(p,r_1)$, the leaf $\hW^{\be}_p$ is $C^1$ and the tangent 
space $T_q\hW^{\be}_p(q)$ lies in a cone of radius $\vep$ about $E^{\be}(q)$.
\item [(ii)] Local invariance. For each $r<r_1/C$ and for each $q\in B(p, r)$,
\[f(\hW^{\be}_p(q,r))\subset \hW^{\be}_{f(p)}(f(q),Cr)\quad\text{and}\quad f^{-1}(\hW^{\be}_p(q,r_1))\subset \hW^{\be}{\be}_{f^{-1}(p)}(f^{-1}(q),Cr).\]
\item [(iii)] Coherence. $\hW^s_p$ and $\hW^c_p$ subfoliate $\hW^{cs}_p$, while $\hW^u_p$ and $\hW^c_p$ subfoliate $\hW^{cu}_p$.
\end{itemize}
\end{proposition}

The basic idea of the proof is the following. By standard arguments, the Bowen balls must lie inside the fake central leaves. Because we consider the time one map of a flow without fixed points, the Bowen balls must be "subfoliated" by flow curves. Then we construct $(n,\ep)$ spanning sets for the Bowen balls by reducing to the one dimensional situation, where we can use the order: we make a selection in the transversal direction, taking a number of flow curves which span the transversal direction, and then we select spanning points on each selected flow curve. In order to deal with the possible shear of the flow curves, we have to select the flow curves to be close enough, and we use again the fact that we have a flow without fixed points.

\begin{proof}[Proof of Theorem \ref{thm:metrusc}]

We denote by $l(\phi_{[0,t]}(x))$ the length of the flow curve from $x$ to $\phi_t(x)$. Without loss of generality, let us assume that the the norm of the vector field $X$ generating $\phi$ is between $1/2$ and $2$, this can be done eventally by changing the Riemannian metric. This implies that for any $t,s\in\mathbb R$ and $x\in M$, we have
$$
\frac t2\leq l(\phi_{[0,t]}(x))\leq 2t
$$
and
$$
\frac 14\leq\frac{l(\phi_{[0,t]}(x))}{l(\phi_{[s,t+s]}(x))}\leq 4.
$$
We can also assume that $r_1$ in Proposition \ref{p:fakefol} is small enough such that for any $y, \phi_t(y)\in\hW^c_x(r_1)$, the different distances $d_M(y,\phi_t(y)),\ d_{\hW^c_x}(y,\phi_t(y))$ and $l(\phi_{[0,t]}(y))$ are comparable up to a constant close to one.

Standard arguments which are used for example in \cite{CY,DFPV,LVY} imply that there exist $0<r_2<\frac{r_1}2$ sufficiently small, such that for all $0<r<r_2$, the Bowen ball $B_{\infty}(x,r)$ must be inside the fake central leaf $\hW^c_x(r_1/2)$, this is because the unstable and stable (fake) foliations are uniformly expanded by $\phi_1$ or $\phi_{-1}$.

{\bf Fake center leaves, flow curves and Bowen balls.} Let $T_x$ be a family of uniformly smooth curves in each $\hW^c_x$ which is uniformly transverse to the vector field $X$. Let us point out that $X$ is not necesarily tangent to $\hW^c_x$, so define the vector field $Y_x$ on $\hW^c_x(r_1)$ to be the orthogonal projection of $X$ to $T\hW^c_x(r_1)$. Eventually making $r_1$ smaller, we can assume that the vector fields $Y_x$ are uniformly bounded away from zero and infinity, uniformly transverse to $T_x$, uniformly $C^1$, and $C^1$ close to $X$. Let $\phi^x$ be the flow generated by the vector field $Y_x$ inside $\hW^c_x(r_1)$. Using this flow we can define (again for a small enough $r_1$) a projection $\pi_x:\hW^c_x(r_1/2)\rightarrow T_x$, where $\pi_x(y)$ is the unique intersection of the flow curve of $\phi^x$ passing through $y$ with $T_x$.

We make the following observation: if $y\in B_{\infty}(x,r)$ then $\phi_t(y)\in B_{\infty}(x,r+2|t|)$. As a consequence, $\phi_{[-t,t]}(B_{\infty}(x,r))\subset B_{\infty}(x,r+2t)$. Let $r<\frac{r_2}{17}$, so we have $\phi_{[-8r,8r]}(B_{\infty}(x,r))\subset B_{\infty}(x,17r)\subset\hW^c_x(r_1/2)$. In this case we have that on $\phi_{[-8r,8r]}(B_{\infty}(x,r))$, $Y_x$ coincides with $X$ and $\phi^x$ coincides with $\phi$. Let $J_x=\pi_x(B_{\infty}(x,r))=\phi_{[-4r,4r]}(B_{\infty}(x,r))\cap T_x\subset T_x(2r)$. Then $J_x\subset\phi_{[-4r,4r]}(B_{\infty}(x,r))$, which is equivalent to $B_{\infty}(x,r)\subset\phi_{[-4r,4r]}(J_x)$, and furthermore
$$
B_{\infty}(x,r)\subset\phi_{[-4r,4r]}(J_x)\subset\phi_{[-8r,8r]}(B_{\infty}(x,r))\subset\hW^c_x(r_1/2).
$$

In other words, we know that the Bowen ball $B_{\infty}(x,r)$ is inside the union of the flow curves $\phi_{[-4r,4r]}(y)$, for $y\in J_x$, where $J_x$ is a subset of the transversal $T_x(4r)$, and the union of these curves is inside the fake center foliation.

{\bf Uniform local flow charts for fake center leaves.}  For every $x\in M$ let $\gamma_x:[-4r,4r]\rightarrow T_x(4r)$ be a parametrization by length. We define the charts $\psi_x:[-4r,4r]^2\rightarrow\phi^x_{[-4r,4r]}(T_x(4r))\subset\hW^c_x(r_1/2)$ by the formula
$$
\psi_x(a,b)=\phi^x_a(\gamma_x(b)), \forall a,b\in[-4r,4r].
$$
In this charts the transversal $T_x(4r)$ is the vertical segment $\{0\}\times[-4r,4r]$, and the flow $\phi^x$ is the horizontal translation. Clearly the charts $\psi_x$ are uniformly $C^1$. Let $\gamma^{-1}(J_x)=I_x$, and since $\phi$ and $\phi^x$ coincide on $\phi_{[-4r,4r]}(J_x)$ we have that $B_{\infty}(x,r)\subset\psi_x([-4r,4r]\times I_x)$.

Analyzing the inverse of $\psi_x$ we get that
$$
\psi_x^{-1}(y)=(t_x(y), \gamma_x^{-1}(\pi_x(y))),
$$
where $t_x(y)\in[-4r,4r]$ is such that $\phi^x_{t_x(y)}(\pi_x(y))=y$, or $\phi^x_{-t_x(y)}(y)=\pi_x(y)$.

Let $f_x:=\psi_{\phi_1(x)}^{-1}\circ\phi_1\circ\psi_x:U_x\rightarrow[-4r,4r]^2$, where $U_x\subset[-4r,4r]^2$ is the maximal domain of $f_x$. Since $\phi_1(B_{\infty}(x,r))=B_{\infty}(\phi_1(x),r)$ we get that $\psi_x^{-1}(B_{\infty}(x,r))\subset U_x$. Remember that the local charts take horizontal segments in $[-4r,4r]\times I_x$ to flow curves for $\phi$, and $\phi_1$ preserves the flow curves. This means that for any $b\in I_x$, the horizontal segment $([-4r,4r]\times\{b\})\cap U_x$ is mapped by $f_x$, as a translation, inside a segment $[-4r,4r]\times\{b'\}$, for some $b'\in I_{\phi_1(x)}$. In other words, there exist $\alpha_x:I_x\rightarrow[-4r,4r]$ and $\beta_x:I_x\rightarrow I_{\phi_1(x)}$ such that, for any $(a,b)\in\psi_x^{-1}(B_{\infty}(x,r))$, we have
$$
f_x(a,b)=(a+\alpha_x(b),\beta_x(b)).
$$

This can also be seen from a direct computation of $f_x$. If $(a,b)\in\psi_x^{-1}(B_{\infty}(x,r))$ then
$$
f_x(a,b)=\left( t_{\phi_1(x)}(\phi_{a+1}(\gamma_x(b))), \gamma_{\phi_1(x)}^{-1}(\pi_{\phi_1(x)}(\phi_{a+1}(\gamma_x(b))))\right).
$$
The second component $\beta_x(b)$ is clearly independent of $a$, since $z=\pi_{\phi_1(x)}(\phi_{a+1}(\gamma_x(b)))$ is independent of $a$. We also have that
$$
\phi_{  t_{\phi_1(x)}(\phi_{a+1}(\gamma_x(b)))}(z)=\phi_{a+1}(\gamma_x(b)),
$$
so $\alpha_x(b)=t_{\phi_1(x)}(\phi_{a+1}(\gamma_x(b)))-a$ depends only on $b$ and not on $a$. It is not difficult to show that $\beta_x$ is in fact invertible, by applying similar arguments to $\phi_{-1}$. In fact we can show more than that, $\beta_x$ must be monotone. In order to show that, we observe that $\beta_x$ is the restriction to $I_x$ of the composition of $\gamma_{\phi_1(x)}^{-1}$, which is clearly monotone, with the map $\tilde\beta_x:[-2r,2r]\rightarrow T_{\phi_1(x)}$,
$$
\tilde\beta_x(b)=\pi_{\phi_1(x)}(\phi_1(\gamma_x(b))).
$$
By choosing $r_1$ (or $r$) sufficiently small, we can guarantee that $\phi_1(T_x(2r))$ is transversal to $Y_{\phi_1(x)}$, and then $\tilde\beta$ does not have critical points, so it is also monotone.

{\bf Construction of spanning sets.} The functions $\psi_x$ are uniformly $C^1$, so they are uniformly Lipschitz. This implies that for any $\vep>0$, there exists $\vep'>0$ such that, for any $x\in M$, $p,q\in[-4r,4r]^2$ with $d(p,q)<\vep'$, then $d(\psi_x(p),\psi_x(q))<\vep$. In other words, if we construct an $(n,\vep')$ spanning set for the sequence $(f_{\phi_n(x)})_{n\geq 0}$, then applying $\psi_x$ we will obtain an $(n,\vep)$ spanning set for $\phi_1$.

Since the charts $\psi_x$ are uniformly $C^1$, we know that the maps $f_x$ are also uniformly $C^1$, in particular they are all Lipschitz for some Lipschitz constant $L>0$. This in turns implies that the functions $\alpha_x$ and $\beta_x$ are all Lipschitz with Lipschitz constant $L$.

Now for every $n\geq 0$, let $S_n$ be a $\delta/2$-spanning set in $I_{\phi_n(x)}$ (we will specify $\delta$ later). The cardinality of $S_n$ is at most $8r/\delta$. Let
$$
S_I=\cup_{i=0}^{n-1}\beta_x^{-1}\circ\beta_{\phi_1(x)}^{-2}\circ\dots\circ\beta_{\phi_{i-1}(x)}^{-1}(S_i).
$$

Then the cardinality of $S_I$ is at most $8nr/\delta$, and $S$ is $(n,\delta)$ spanning on $I_x$ for the sequence of maps $(\beta_{\phi_n(x)})_{n\geq 0}$, in the sence that for any $b\in I_x$, there exists $b'\in S_I$ such that $|\beta_{\phi_{i-1}(x)}\circ\dots\circ\beta_x(y)-\beta_{\phi_{i-1}(x)}\circ\dots\circ\beta_x(y')|<\delta$ for all $0\leq i\leq n-1$.

Let $S_h$ be an $\vep'/2$ spanning set in $[-4r,4r]$, with cardinality $16r/\vep'$, and let $S=S_h\times S_I$ inside $[-4r,4r]\times I_x$. Observe that the cardinality of $S$ is at most $\frac{128nr}{\delta\vep'}$. We claim that, for $\delta$ small enough, $S$ is a $(n,\vep')$ spanning set of $\psi_x^{-1}(B_{\infty}(x,r))$ for the sequence of maps $f_{\phi_i(x)}$, $0\leq i\leq n-1$.

Let us use the notation $\beta^k=\beta_{\phi_{k-1}(x)}\circ\dots\circ\beta_x$, $f^n=f_{\phi_{k-1}(x)}\circ\dots\circ f_x$, for all $k\geq 1$. By the choice of $S=S_h\times S_I$, one can see that given any $(a,b)\in\psi_x^{-1}(B_{\infty}(x,r))$, there exists an $(a',b')\in S$ such that
\begin{eqnarray*}
|a-a'|&<&\frac{\vep'}2;\\
|\beta^i(b)-\beta^i(b')|&<&\delta,\ \ \forall 0\leq i\leq n-1.
\end{eqnarray*}

Then we get that
\begin{eqnarray*}
|f^i(a,b)-f^i(a',b')|&=&|\beta^{i-1}(b)-\beta^{i-1}(b')|+|(a-a')+(\alpha_x(b)-\alpha_x(b'))+\dots \\
&&\dots +(\alpha_{\phi_{i-1}(x)}(\beta^{i-2}(b))-\alpha_{\phi_{i-1}(x)}(\beta^{i-2}(b')))|\\
&<&\delta+\frac{\vep'}2+(i-1)L\delta.
\end{eqnarray*}

If we choose $\delta=\frac{\vep'}{2(1+nL)}$ then we obtain that $|f^i(a,b)-f^i(a',b')|<\vep'$, for all $0\leq i\leq n-1$, so indeed $S$ is $(n,\vep')$ spanning. As we remarked before, this will imply that $\psi_x(S)$ is $(n,\vep)$ spanning for $B_{\infty}(x,r)$ and $\phi_1$, and it has the cardinality bounded from above by
$$
|\psi_x(S)|=|S|\leq\frac{128nr}{\delta\vep'}=\frac{256n(1+nL)r}{\vep'^2}.
$$

Since $\lim_{n\rightarrow\infty}\frac 1n\log\frac{256n(1+nL)r}{\vep'^2}=0$, this allows us to conclude that the local entropy vanishes, so $\phi_1$ is entropy expansive.

\end{proof}

%%%%%%%%%%%%%%%%%%%%%%%%%%%%%%%%%%%%%%%%%%%%%%%%%

\bibliographystyle{plain}
\def\cprime{$'$}
% \bib, bibdiv, biblist are defined by the amsrefs package.

\bigskip
\bigskip

\noindent {\scshape Mario Rold\'an}\\
Departamento de Matem\'atica, Universidade Federal de Santa Catarina\\
Campus Universit\'ario Trindade,\\
Florian\'opolis - SC - Brazil CEP 88.040-900 \\
\texttt{roldan@impa.br}

\medskip

\noindent {\scshape Radu Saghin}\\
Instituto de Matem\'atica, Pontificia Universidad Cat\'olica de Valpara\'iso\\
Blanco Viel 596, Cerro Bar\'on, Valpara\'iso, Chile.\\
\texttt{rsaghin@gmail.com}

\medskip

\noindent {\scshape Jiagang Yang}\\
Department of Mathematics, Southern University of Science and Technology of China, 1088 Xueyuan Rd., Xili, Nanshan District, Shenzhen, Guangdong, China 518055\\
Instituto de Matem\'atica e Estat\'istica, Universidade Federal Fluminense\\
Rua M\'ario Santos Braga s/n - Campus Valonguinhos. Niter\'oi, Brazil\\
\texttt{yangjg@impa.br}

\end{document}